\documentclass[11pt]{article}

\usepackage[utf8]{inputenc}
\usepackage[english]{babel}
\usepackage{hyperref}

\usepackage{amssymb,amsmath,amsthm}
\usepackage{enumerate}
\usepackage{enumitem}
\usepackage{algorithm}
\usepackage{algorithmic}
\usepackage{array}
\usepackage{graphicx, color}
\usepackage{epsfig}
\usepackage{tikz}
\usetikzlibrary{calc,decorations.pathmorphing,decorations.text}
\usepackage{subcaption}
\usepackage{refcount} 
\usepackage[hmargin=2.5cm,vmargin=3cm]{geometry}
\usepackage{mathtools, braket}
\usepackage[normalem]{ulem}
\usepackage{authblk}
\usepackage{centernot}
\usepackage{indentfirst}
\usepackage{framed} 
\usepackage{amsmath}
\allowdisplaybreaks[4]
\usepackage[nameinlink]{cleveref}

\usepackage{hyperref}
\hypersetup{
    hidelinks,
	colorlinks=true, 
	linkcolor=blue,
	citecolor=red}

\theoremstyle{plain}
\newtheorem{thm}{Thm}[section]

\newtheorem{theorem}[thm]{Theorem}
\newtheorem{lemma}[thm]{Lemma}
\newtheorem{corollary}[thm]{Corollary}

\newtheorem{conjecture}[thm]{Conjecture}
\newtheorem{problem}[thm]{Problem}

\newtheorem{observation}[thm]{Observation}
\newtheorem{definition}[thm]{Definition}

\usepackage{bm}

\usepackage[nameinlink]{cleveref}

\newtheorem{innercustomgeneric}{\customgenericname}
\providecommand{\customgenericname}{}
\newcommand{\newcustomtheorem}[2]{%
	\newenvironment{#1}[1]
	{%
		\renewcommand\customgenericname{#2}%
		\renewcommand\theinnercustomgeneric{##1}%
		\innercustomgeneric
	}
	{\endinnercustomgeneric}
}

\newcustomtheorem{customtheorem}{Theorem}
\newcustomtheorem{customlemma}{Lemma}

\makeatletter
\newtheoremstyle{manual}
  {\topsep}{\topsep}
  {\itshape}{}
  {\bfseries}{.}{ }
  {#1 #3}
\makeatother

\theoremstyle{manual}
\newtheorem*{manualthm}{Theorem}

\newenvironment{proof*}{\noindent \emph{Proof.}}{\hfill$\Diamond$}

\renewcommand{\pod}[1]{\allowbreak\mathchoice
	{\if@display \mkern 0mu\else \mkern 0mu\fi (#1)}
	{\if@display \mkern 0mu\else \mkern 0mu\fi (#1)}
	{\mkern 1mu(\mathrm{mod}\mkern 4mu #1)}
	{\mkern 0mu(#1)}
}

\usepackage[color=green!30]{todonotes}

\usepackage{caption}
\usetikzlibrary{matrix}
\usetikzlibrary{decorations.markings}
\tikzstyle{vertex}=[circle, draw, fill=black!50,
inner sep=0pt, minimum width=4pt]
\tikzset{->-/.style={decoration={
			markings,
			mark=at position .5 with {\arrow{>}}},postaction={decorate}}}

\tikzstyle{bigblue}=[color=blue, very thick, >=stealth]
\tikzstyle{lightblue}=[color=blue, thin, >=stealth]

\tikzstyle{bigred}=[color=red, very thick, >=stealth]
\tikzstyle{lightred}=[color=red, thin, >=stealth]

\tikzstyle{biggreen}=[color=black!30!green, very thick, >=stealth]
\tikzstyle{lightgreen}=[color=black!30!green,  thin, >=stealth]

\def\R{\mathbb{R}}

\usepackage{amsmath, amssymb, array, booktabs, caption}

\newcolumntype{M}[1]{>{\centering\arraybackslash}m{#1}}

\newcommand{\bs}[1]{\boldsymbol{#1}}

\title{High-Dimensional $p$-Normed Flows}
\author{Chenxing Li$^1$, Jiaao Li$^1$, Rong Luo$^2$, and Bo Su$^1$\\
\small $^1$School of Mathematical Sciences and LPMC, Nankai University, Tianjin 300071, China \\
\small $^2$Department of Mathematics, West Virginia University, Morgantown 26505, USA\\
\small Emails: chenxingli@mail.nankai.edu.cn; lijiaao@nankai.edu.cn; rluo@mail.wvu.edu; suboll@mail.nankai.edu.cn
}
\date{}

\begin{document}

\maketitle

\begin{abstract}
We generalize Tutte's integer flows and the $d$-dimensional Euclidean flows of Mattiolo, Mazzuoccolo, Rajn\'{i}k, and Tabarelli to \emph{$d$-dimensional $p$-normed nowhere-zero flows} and define the corresponding flow index $\phi_{d,p}(G)$ to be the infimum over all real numbers $r$ for which $G$ admits a $d$-dimensional $p$-normed nowhere-zero $r$-flow.
For any bridgeless graph $G$ and any $p\ge 1$, we establish general upper bounds, including $\phi_{2,p}(G) \le 3$, $\phi_{3,p}(G) \le 1+\sqrt{2}$, and tight bounds for graphs admitting a $4$-NZF.
For graphs with oriented $(k+1)$-cycle $2l$-covers, we show that $\phi_{k,p}(G) = 2$, which implies $\phi_{2,p}(G) = 2$ for graphs admitting a nowhere-zero $3$-flow and $\phi_{3,p}(G) = 2$ for those admitting a nowhere-zero $4$-flow.
These results extend classical flow theory to arbitrary norms, provide supporting evidences for Tutte's $5$-flow Conjecture and Jain's $S^2$-Flow Conjecture, and connect combinatorial flows with geometric and topological perspectives.

\end{abstract}

\section{Introduction}

Graphs considered in this paper are finite and may have multiple edges but no loops.
For terminology and notation not defined here, we follow \cite{BM08, zhang1997integer}.

\subsection{Background, motivation, and definitions}

Tutte~\cite{Tutte54, Tutte1949} introduced the theory of integer flows as a dual problem
to vertex colorings of planar graphs. The concept of integer flows was later extended
to real-valued flows by Goddyn, Tarsi, and Zhang~\cite{Goddyn}, and further generalized
to vector flows by Jain~\cite{devos} and Thomassen~\cite{thomassen2014group}.
These developments provide a geometric framework for studying classical flow
conjectures, most notably Tutte’s $5$-Flow Conjecture.

Let $D$ be an orientation of a graph $G$. For a vertex $v$, let $E_D^+(v)$ and
$E_D^-(v)$ denote the sets of edges directed out of and into $v$, respectively.

\begin{definition}
Let $\Omega \subseteq \mathbb{R}^d$ and let $G$ be a graph with an orientation $D$.
An ordered pair $(D,f)$ is a \emph{vector $\Omega$-flow} of $G$ if
$f \colon E(G)\to \Omega$ satisfies
\[
\sum_{e\in E_D^+(v)} f(e)-\sum_{e\in E_D^-(v)} f(e)=\bm{0}
\]
for every vertex $v\in V(G)$.
\end{definition}

If the context is clear, we simply write \emph{$\Omega$-flow}.
Vector flows were first studied for $\Omega=S^{d-1}$ in
\cite{thomassen2014group, wang2015vector}, where
\(
S^{d-1}=\{\bs{x}\in\mathbb{R}^d:\|\bs{x}\|_2=1\}.
\)
Mattiolo, Mazzuoccolo, Rajn\'{i}k, and Tabarelli~\cite{mattiolo2023d} further introduced $d$-dimensional $r$-flows,
allowing flow values in $\mathbb{R}^d$ under the Euclidean norm.
In this paper, we extend this framework to arbitrary $p$-norms.

Throughout this paper, ``$\infty$'' always denotes $+\infty$.  
For $a \in \mathbb{R}$, we define
\[
[a,\infty] := \{ x \in \mathbb{R} : x \ge a \} \cup \{\infty\}, \qquad
(a,\infty] := \{ x \in \mathbb{R} : x > a \} \cup \{\infty\}.
\]
For functions defined on intervals of the form $[a,\infty)$ or $(a,\infty)$,
we interpret $f(\infty)$ as $\lim_{x\to\infty} f(x)$ whenever the limit exists.
In particular, for $\bs{x}=(x_1,\dots,x_d)\in\mathbb{R}^d$,
\[
\|\bs{x}\|_p=\bigl(|x_1|^p+\cdots+|x_d|^p\bigr)^{1/p}, \quad p\in[1,\infty),
\qquad
\|\bs{x}\|_\infty=\max_i |x_i|.
\]

\begin{definition}\label{def:annular-flow}
Let $d \ge 1$ be an integer, $r \ge 2$, and $p \in [1, \infty]$.  
A \emph{$d$-dimensional $p$-normed nowhere-zero $r$-flow} of a graph $G$,
abbreviated as an $(r,d,p)$-NZF, is a vector $(\mathbb{R}^d\setminus\{\bm{0}\})$-flow
$(D,f)$ such that
\[
1\le \|f(e)\|_p \le r-1 \quad \text{for every } e\in E(G).
\]
\end{definition}

The \emph{$d$-dimensional $p$-normed flow index} of a bridgeless graph $G$,
denoted $\phi_{d,p}(G)$, is the infimum over all real numbers $r$ for which
$G$ admits an $(r,d,p)$-NZF.
It is immediate that $\phi_{d+1,p}(G)\le \phi_{d,p}(G)$ for all $d$ and $p$. As observed in \cite{mattiolo2023d} for the Euclidean case, $\phi_{d,2}(G)$ is in fact a minimum; the same argument shows that $\phi_{d,p}(G)$ is also a minimum.

When $d=1$, this notion reduces to circular flows, and we write $\phi_1(G)$. For a positive integer $k$, a $(k,1,p)$-NZF is exactly a classical nowhere-zero $k$-flow, abbreviated as a $k$-NZF. The celebrated Tutte’s $5$-Flow Conjecture states the following. 
\begin{conjecture}[Tutte’s $5$-Flow Conjecture]
$\phi_{1}(G)\le 5$ for every bridgeless graph $G$.
\end{conjecture}
In 1981, Seymour~\cite{seymour1981} established the $6$-flow theorem by proving $\phi_1(G)\le6$  for every bridgeless graph.

An $S^{d-1}$-flow is a $d$-dimensional $2$-normed flow.
Jain {\color{red}\cite{devos}} proposed the following two conjectures, which together imply Tutte's $5$-Flow Conjecture.
\begin{conjecture}[Jain’s $S^2$-Flow Conjecture]
$\phi_{3,2}(G)=2$ for every bridgeless graph $G$.
\end{conjecture}
\begin{conjecture}[Jain's $S^2$-Map Conjecture]
    There exists a map $\varphi:S^2 \rightarrow \{-4,-3,-2,-1,1,2,3,4\}$ so that antipodal points of $S^2$ receive opposite values, and so that the three vertices of any inscribed equilateral triangle on a great circle have images whose sum is zero.
\end{conjecture}
Both Tutte's and Jain's conjectures are best possible if true as they can be evident by the Petersen graph. In view of Jain's conjectures,  the study of $\phi_{3,2}(G)$ and, more generally, of $d$-dimensional $p$-normed flows can be regarded as an approach to attack Tutte's $5$-Flow Conjecture.

\subsection{Main results: upper bounds for flow indices}
Motivated by Jain’s conjecture, the study of $\phi_{d,p}(G)$ provides a geometric
approach to classical flow problems.
Mattiolo, Mazzuoccolo, Rajn\'{i}k, and Tabarelli~\cite{mattiolo2023d} studied
$2$-dimensional $2$-normed flows and proved the following.

\begin{theorem}[Mattiolo et~al.~\cite{mattiolo2023d}]
\label{thm:mattiolo-known-bounds}
Let $G$ be a bridgeless graph. Then we have
\begin{enumerate}[label=(\arabic*)]
\item $\phi_{2,2}(G)\le1+\sqrt5$;
\item if $G$ admits a $4$-NZF, then $\phi_{2,2}(G)\le1+\sqrt2$.
\end{enumerate}
\end{theorem}

In this paper, we study $2$- and $3$-dimensional $p$-normed flows for all $p$ and
establish the following general bounds.

\begin{theorem}\label{thm:general-2-3-bounds}
Let $G$ be a bridgeless graph. Then we have 
\begin{enumerate}[label=(\arabic*)]
\item $\phi_{2,p}(G)\le3$ for every $p\in[1,\infty]$;
\item $\phi_{3,p}(G)\le1+2^{1/p}$ for every $p\in[1,\infty]$, and in particular, $\phi_{3,\infty}(G)=2$;
\item $\phi_{3,1}(G)\le\frac94$.
\end{enumerate}
\end{theorem}

Note that the upper bound $1+2^{1/p}$ in Theorem~\ref{thm:general-2-3-bounds}(2) is at least $1+\sqrt{2}$ for all $p\in[1,2)$, and is therefore somewhat loose. Indeed, we expect $\phi_{3,p}(G)$ to be close to $2$.
To state improved bounds, define
\[
g_1(p)=\frac54\cdot2^{\,1-1/p}, \qquad
g_2(p)=\Bigl[(2^{-1/p}+3^{-1/p})^p+(2^{-1/p}-3^{-1/p})^p+\tfrac13\Bigr]^{1/p},
\qquad
g_3(p)=2^{1/p}.
\]

\begin{theorem}\label{thm:3-dim-p-12-flow}
Let $G$ be a bridgeless graph. For any $p\in[1,\infty]$,
\[
\phi_{3,p}(G)\le 1+\min\{g_1(p),\,g_2(p),\,g_3(p)\}\le 1+\sqrt2.
\]
\end{theorem}

Further discussion and numerical comparisons among $g_1$, $g_2$, and $g_3$ are deferred to the final section. In particular, Theorem \ref{thm:3-dim-p-12-flow} provides an upper bound toward Jain's $S^2$-Flow Conjecture by showing that $\phi_{3,2}(G) \le 1 + \sqrt{2}$ for every bridgeless graph $G$.

When $G$ admits a $4$-NZF, we establish a sharp bound.
A graph $H$ is \emph{contractible} to a graph $G$ if $G$ can be obtained from $H$
by repeatedly identifying vertices and deleting resulting loops.

\begin{theorem}\label{thm:4-flow-2d-bound}
Let $G$ be a graph admitting a $4$-NZF. Then
\[
\phi_{2,p}(G)\le
\begin{cases}
1+2^{\,1-1/p}, & 1\le p\le2,\\[1mm]
1+2^{\,1/p}, & p\in[2,\infty].
\end{cases}
\]
In particular, $\phi_{2,1}(G)=\phi_{2,\infty}(G)=2$.
The bound is attained for every $p$ by any graph admitting a $4$-NZF and contractible to $K_4$.
\end{theorem}

\subsection{Main results: connections with cycle covers}

Vector flows are closely related to cycle covers, extending the well-known
connections between integer (or circular) flows and cycle decompositions.
We begin by recalling the necessary definitions.

A \emph{circuit} of a graph $G$ is a connected $2$-regular subgraph.
A \emph{cycle} is a subgraph of $G$ in which every vertex has even degree.
Let $\mathcal{C}$ be a family of cycles in $G$.
We say that $\mathcal{C}$ is a \emph{cycle cover} of $G$ if every edge of $G$
is contained in at least one member of $\mathcal{C}$.
Moreover, $\mathcal{C}$ is called an \emph{$s$-cycle $t$-cover} if it consists of
$s$ cycles and each edge of $G$ is contained in exactly $t$ members of $\mathcal{C}$.
When $t=2$, $\mathcal{C}$ is referred to as a \emph{double cover}.

Bermond, Jackson, and Jaeger~\cite{bermond} proved that every bridgeless graph admits
a $7$-cycle $4$-cover, which implies that $\phi_{d,p}(G)=2$ for every integer $d\ge7$
and every $p\in[1,\infty]$.
The well-known $5$-Cycle Double Cover Conjecture of Celmins and Preissmann would
imply $\phi_{d,p}(G)=2$ for every integer $d\ge5$.

A stronger notion is that of an \emph{oriented cycle cover}.
A cycle in a directed graph is said to be \emph{directed} if every vertex has equal
indegree and outdegree.
A family $\mathcal{D}$ of $k$ directed cycles is called an
\emph{oriented $k$-cycle $2l$-cover} of a graph $G$ if each edge of $G$ is covered
exactly $2l$ times by $\mathcal{D}$, with $l$ occurrences in each direction.
For brevity, such a cover is denoted a \emph{$(k,2l)$-OCC}.
When $l=1$, the cover is also called an \emph{oriented cycle double cover}.

The Oriented Berge--Fulkerson Conjecture, proposed by Ulyanov~\cite{Ulyanov-url, Ulyanov2025},
asserts that every bridgeless graph admits a $(6,4)$-OCC.
The Oriented $5$-Cycle Double Cover Conjecture, proposed by Archdeacon and
Jaeger~\cite{archdeacon1984face, Jaeger1988}, asserts that every bridgeless graph
admits a $(5,2)$-OCC.
Mattiolo et~al.~\cite{mattiolo2023d, mattiolo2025geometric} showed that
$\phi_{2,2}(G)\le \frac{3+\sqrt5}{2}$ when $G$ admits a $(5,2)$-OCC, and that
$\phi_{k,2}(G)=2$ whenever $G$ admits an oriented $(k+1)$-cycle double cover.

In this paper, we extend these results to general $p$-norms.
Our main result in this direction is the following.

\begin{theorem}\label{thm:intro-k-OCDC-phi}
If a graph $G$ admits a $(k,2l)$-OCC, then we have
\begin{enumerate}[label=(\arabic*)]
\item $\phi_{d,p}(G)=2$ for every $p\in[1,\infty]$, where $d=k-1$;
\item if $l=1$, then $\phi_{d,1}(G)=2$, where $d=\lceil k/2\rceil$.
\end{enumerate}
\end{theorem}

Theorem~\ref{thm:intro-k-OCDC-phi} immediately yields the following corollary.

\begin{corollary}\label{cor}
Let $G$ be a bridgeless graph.
\begin{enumerate}[label=(\arabic*)]
\item If the Oriented $5$-Cycle Double Cover Conjecture holds, then
$\phi_{4,p}(G)=2$ for every $p\in[1,\infty]$ and $\phi_{3,1}(G)=2$.
\item If the Oriented Berge--Fulkerson Conjecture holds, then
$\phi_{5,p}(G)=2$ for every $p\in[1,\infty]$.
\end{enumerate}
\end{corollary}

For $k=3,4$, a graph admits an oriented $k$-cycle double cover if and only if it
admits a $k$-NZF (see~\cite{zhang1997integer}).
As a consequence, Theorem~\ref{thm:intro-k-OCDC-phi} implies the following.

\begin{corollary}\label{Cor:3-4-flow}
Let $G$ be a bridgeless graph.
\begin{enumerate}[label=(\arabic*)]
\item If $G$ admits a $3$-NZF, then $\phi_{2,p}(G)=2$ for every $p\in[1,\infty]$.
\item If $G$ admits a $4$-NZF, then $\phi_{3,p}(G)=2$ for every $p\in[1,\infty]$.
\end{enumerate}
\end{corollary}

After the completion of this paper, Mazzuoccolo shared with us the references
\cite{gaborik2025manhattan} and \cite{Rieg-Master}. He informed us that the
equality $\phi_{3,\infty}(G)=2$ in Theorem~\ref{thm:general-2-3-bounds}(2) and the equality
$\phi_{3,1}(G)=2$ in Corollary~\ref{cor}(1) were obtained
independently in \cite{gaborik2025manhattan}. He also pointed out that the upper
bounds in Theorem~\ref{thm:4-flow-2d-bound} (without the sharpness statement) and
the case $l=1$ of Theorem~\ref{thm:intro-k-OCDC-phi}(1) were obtained independently
in \cite{Rieg-Master}. 

The organization of the rest of the paper is as follows. In Section~2, we develop a general approach for deriving upper bounds on $\phi_{d,p}(G)$; as a direct application, this section also contains the proof of Theorem~\ref{thm:4-flow-2d-bound}.
The proofs of Theorems~\ref{thm:general-2-3-bounds} and \ref{thm:3-dim-p-12-flow} are given in Section~3, and the proof of Theorem \ref{thm:intro-k-OCDC-phi} is presented in Section 4. Further discussion, along with open problems and conjectures, will be presented in Section~5.

\section{A general method for finding upper bounds of flow indices}
In this section, we present a general method for deriving upper bounds on
$\phi_{d,p}(G)$ and, as an application, use it to prove
Theorem~\ref{thm:4-flow-2d-bound}.

Let $P_1,\dots,P_t$ be formal vectors in $\mathbb{R}^d$, viewed for the moment as symbolic variables. Let $\Omega$ be a prescribed subset of $\mathbb{R}^d$ generated by these vectors (for example, $\Omega=\{P_1,P_2,P_1+P_2\}$). 
Consider the following optimization problem:
\begin{equation}\label{eq:optimization}
\begin{aligned}
& \min_{P_1,\dots,P_t} && \max\{\|\bs{x}\|_p : \bs{x}\in\Omega\},\\
& \text{s.t.} && \|\bs{x}\|_p \ge 1 \quad \forall\,\bs{x}\in\Omega.
\end{aligned}
\end{equation}
If a graph $G$ admits an $\Omega$-flow, then any numerical assignment of
$P_1,\dots,P_t$ such that every $\bs{x}\in\Omega$ satisfies
\[
1 \le \|\bs{x}\|_p \le r-1
\]
immediately yields $\phi_{d,p}(G)\le r$. 
Therefore, for graphs admitting an $\Omega$-flow, bounding $\phi_{d,p}(G)$ reduces to finding feasible solutions to \eqref{eq:optimization} with the objective value as small as possible. 
This principle is summarized in the following lemma.

\begin{lemma}[A general method for upper bounds on $\phi_{d,p}(G)$]
\label{lem:Omega-flow-phi}
Let $d\ge 2$ and $p\in[1,\infty]$. 
Suppose that a bridgeless graph $G$ admits an $\Omega$-flow for some vector set $\Omega\subseteq\mathbb{R}^d$ generated by formal vectors $P_1,\dots,P_t$.
If there exists a numerical assignment $P_1,\dots,P_t\in\mathbb{R}^d$ such that
$
1 \le \|\bs{x}\|_p \le r-1$  {for all} $\bs{x}\in\Omega$,
then
\[
\phi_{d,p}(G)\le r.
\]
\end{lemma}

Accordingly, bounding $\phi_{d,p}(G)$ naturally separates into two steps:

\begin{enumerate}[label=(\arabic*)]
\item \textbf{Identify a symbolic structure for $\Omega$.}
One aims to select a set $\Omega$ that both arises naturally from the flow
constraints and is amenable to optimization.

\item \textbf{Choose concrete vectors.}
After fixing $\Omega$, one chooses numerical vectors $P_1,\dots,P_t$ such that $\|\bs{x}\|_p \ge 1$ for all $\bs{x}\in\Omega$, while making
$\max\{\|\bs{x}\|_p:\bs{x}\in\Omega\}$ as small as possible, thereby yielding an explicit
upper bound.
\end{enumerate}

For a fixed graph $G$, a convenient way to construct such a symbolic set is to choose a spanning tree $T$, assign independent symbolic variables to the edges in $E(G)\setminus E(T)$, and use flow conservation to express the value on each tree edge as a linear combination of these variables. This yields a canonical symbolic set $\Omega$ for which $G$ admits an $\Omega$-flow. 
This method can be applied in compute programming to provide the bounds $\phi_{1}(P_{10})\le 5$, $\phi_{3,2}(P_{10})=2$, and $\phi_{2,2}(P_{10})\le 1+\sqrt{7/3}$ (which is obtained in \cite{mattiolo2023d} using a different approach) for the Petersen graph $P_{10}$.

For general graph families (e.g., all bridgeless graphs), however, the
spanning-tree construction is not applicable, as it depends on the structure of a
specific graph. Instead, one seeks symbolic sets $\Omega$ that arise uniformly
across the family and can be analyzed independently of any particular instance.
Developing such universal symbolic structures is therefore essential. In this
paper, we focus on symbolic sets $\Omega$ derived uniformly from structural
properties of $k$-NZFs, enabling bounds for broad classes of graphs.

\begin{lemma}\label{thm:234-flow}
Let $G$ be a bridgeless graph and let $d \ge 2$ be an integer. 
Given any nonzero vector $P \in \mathbb{R}^d$ and any two non-collinear vectors $P_1, P_2 \in \mathbb{R}^d$, the following statements hold:
\begin{enumerate}[label=(\arabic*)]
    \item $G$ admits a $2$-NZF if and only if it admits a $\{P\}$-flow.
    \item $G$ admits a $3$-NZF if and only if it admits a $\{P_1, P_2, P_1 + P_2\}$-flow.
    \item $G$ admits a $4$-NZF if and only if it admits a $\{P_1, P_2, P_1 + P_2, P_1 - P_2\}$-flow.
\end{enumerate}
\end{lemma}

\noindent\textbf{Remark.} In the implication from the existence of a $k$-NZF to the corresponding higher-dimensional flow, 
the conditions $P \neq \mathbf{0}$ and the non-collinearity of $P_1$ and $P_2$ are not strictly necessary; 
the construction remains valid for any choice of vectors in $\mathbb{R}^d$.

\begin{proof}
(1) is immediate since both conditions are equivalent to $G$ being a cycle.

\medskip
(2) Little, Tutte, and Younger \cite{little1988theorem} (see also Theorem~2.6.2 in \cite{zhang1997integer}) showed that a graph $G$ admits a $k$-NZF $(D, f)$ with $f(e) > 0$ for each edge $e$ if and only if there exist $k-1$ directed cycles $\mathcal{C} = \{C_1, \dots, C_{k-1}\}$ in the oriented graph $D(G)$ such that each edge $e$ lies in exactly $f(e)$ cycles in $\mathcal{C}$.  

Consequently, if $G$ admits a $3$-NZF, then $G$ admits a positive flow $(D, f)$ where $f(e) = 1$ or $2$.  
Under the orientation $D$, there are two directed cycles $C_1$ and $C_2$, such that each edge is covered exactly $f(e)$ times.  
Let $P_1, P_2 \in \mathbb{R}^d$. By (1), $G$ has a $\{P_i, 0\}$-flow $g_i$ for each $i = 1,2$, where $g_i(e) = P_i$ if and only if $e \in E(C_i)$.  
For each edge $e \in C_1 \cap C_2$, $g_1(e) + g_2(e) = P_1 + P_2$. Therefore, $g_1 + g_2$ is a $\{P_1, P_2, P_1 + P_2\}$-flow.  

Conversely, suppose that $G$ admits a $\{P_1, P_2, P_1 + P_2\}$-flow, where $P_1$ and $P_2$ are non-collinear.  
Let $E_1$ be the set of edges whose flow values are in $\{P_1, P_1 + P_2\}$, and $E_2$ the set of edges whose flow values are in $\{P_2, P_1 + P_2\}$.  
Since $P_1$ and $P_2$ are non-collinear, any linear combination $a P_1 + b P_2 = 0$ implies $a = b = 0$.  
Consequently, $E_1$ forms a directed cycle $C_1$ and $E_2$ forms a directed cycle $C_2$, with consistent orientations on $C_1 \cap C_2$.  
It follows that $G$ admits a $3$-NZF.

\medskip
(3) Suppose that $G$ admits a $4$-NZF, which is equivalent to the existence of two cycles $C_1$ and $C_2$ covering all edges of $G$. For each $i=1,2$,   by (1), let $(D_i,f_i)$ be a $\{P_i\}$-flow of $C_i$. 

 We construct an $\Omega$-flow $(D,f)$ of $G$ as follows.  
For each edge $e \in E(G)$, define its orientation $D(e)$ and flow value $f(e)$ as follows:
\[
D(e) =
\begin{cases}
D_1(e), & e \in E(C_1),\\[1mm]
D_2(e), & e \in E(C_2) \setminus E(C_1),
\end{cases}
\] and 
\[
f(e) =
\begin{cases}
f_1(e), & e \in E(C_1) \setminus E(C_2),\\[1mm]
f_2(e), & e \in E(C_2) \setminus E(C_1),\\[1mm]
f_1(e) + f_2(e), & e \in E(C_1) \cap E(C_2),\ D_1(e) = D_2(e),\\[1mm]
f_1(e) - f_2(e), & e \in E(C_1) \cap E(C_2),\ D_1(e) \neq D_2(e).
\end{cases}
\]

Then clearly  $(D,f)$ is  an $\Omega$-flow where $ \Omega=\{P_1, P_2, P_1 + P_2, P_1 - P_2\}$.  Thus  a $4$-NZF naturally induces a $\{P_1, P_2, P_1 + P_2, P_1 - P_2\}$-flow on $G$.

Conversely, suppose that $G$ admits a $\{P_1, P_2, P_1 + P_2, P_1 - P_2\}$-flow, with $P_1$ and $P_2$ non-collinear. Let $E_1$ be the set of edges assigned values in $\{P_1, P_1 + P_2, P_1 - P_2\}$ and $E_2$ the set assigned values in $\{P_2, P_1 + P_2, P_1 - P_2\}$. By non-collinearity of $P_1$ and $P_2$, $E_1$ forms a cycle $C_1$ and $E_2$ forms a cycle $C_2$, which together cover all edges of $G$. Therefore, $G$ admits a $4$-NZF.
\end{proof}

Seymour's $6$-flow theorem \cite{seymour1981}  can be restated as follows.
\begin{theorem}[Seymour \cite{seymour1981}]\label{ob:6-flow-decomposition}
Let $G$ be a bridgeless graph. Then there exist two subgraphs $G_1$ and $G_2$ of $G$ such that \(E(G_1) \cup E(G_2) = E(G)\), \(G_1\) admits a $3$-NZF, and \(G_2\) admits a $2$-NZF.
\end{theorem}

Building on these constructions, we aim to identify a suitable set of vectors $\Omega$ such that every  bridgeless graph  admits an $\Omega$-flow.  
The approach outlined in Lemma~\ref{lem:Omega-flow-phi} can then be applied to potentially derive general upper bounds for higher-dimensional flows across all bridgeless graphs.

For any two subsets $A,B \subseteq \mathbb{R}^d$, define
\[
A + B = \{x + y : x \in A,\, y \in B\} \quad\text{and}~~
A - B = \{x - y : x \in A,\, y \in B\}.
\]

\begin{lemma}\label{thm:6-flow-points}
Let $G$ be a bridgeless graph and let $d \ge 2$ be an integer.  
For any three vectors $P_1, P_2, P_3 \in \mathbb{R}^d$, define  
\[
\Omega_1 = \{P_1, P_2, P_1 + P_2\}, \quad 
\Omega_2 = \{P_3\}, \quad
\Omega_3 = (\Omega_1 + \Omega_2) \cup (\Omega_1 - \Omega_2),
\]
and let $\Omega = \Omega_1 \cup \Omega_2 \cup \Omega_3$.  
Then $G$ admits an $\Omega$-flow.
\end{lemma}

\begin{proof}
Since $G$ is bridgeless, by Theorem~\ref{ob:6-flow-decomposition}, there exist two subgraphs $G_1$ and $G_2$ of $G$ such that $G_1$ admits a $3$-NZF, $G_2$ admits a $2$-NZF, and $E(G) = E(G_1) \cup E(G_2)$.  
By Lemma~\ref{thm:234-flow} (1) and (2), $G_1$ admits an $\Omega_1$-flow $(D_1, f_1)$ and $G_2$ admits an $\Omega_2$-flow $(D_2, f_2)$.

We construct an $\Omega$-flow $(D,f)$ of $G$ as follows.  
For each edge $e \in E(G)$, define its orientation $D(e)$ and flow value $f(e)$ as follows:
\[
D(e) =
\begin{cases}
D_1(e), & e \in E(G_1),\\[1mm]
D_2(e), & e \in E(G_2) \setminus E(G_1),
\end{cases}
\] and 
\[
f(e) =
\begin{cases}
f_1(e), & e \in E(G_1) \setminus E(G_2),\\[1mm]
f_2(e), & e \in E(G_2) \setminus E(G_1),\\[1mm]
f_1(e) + f_2(e), & e \in E(G_1) \cap E(G_2),\ D_1(e) = D_2(e),\\[1mm]
f_1(e) - f_2(e), & e \in E(G_1) \cap E(G_2),\ D_1(e) \neq D_2(e).
\end{cases}
\]
By the construction, each flow value $f(e)$ lies in $\Omega$, and the flow conservation law holds at every vertex.  
Hence, $(D,f)$ is an $\Omega$-flow of $G$.
\end{proof}

As a consequence of Lemma~\ref{thm:6-flow-points} and Lemma~\ref{lem:Omega-flow-phi}, estimating $\phi_{d,p}(G)$ for bridgeless graphs reduces to the following optimization problem: choose vectors $P_1, P_2, P_3 \in \mathbb{R}^d$ to minimize the maximum $p$-norm among the vectors in
\[
\Omega = \{P_1, P_2, P_1+P_2, P_3, P_1+P_3, P_1-P_3, P_2+P_3, P_2-P_3, P_1+P_2+P_3, P_1+P_2-P_3\},
\]
subject to the constraint $\|\bs{x}\|_p \ge 1$ for all $\bs{x} \in \Omega$. In the next section, we provide general constructions for $P_1, P_2, P_3$ for various choices of dimension $d$ and norm $p$.

As an illustrative example, we first present the proof of Theorem~\ref{thm:4-flow-2d-bound} for graphs admitting $4$-NZFs.  
We begin with the following observation, which holds naturally in two-dimensional space. Although we construct $\Omega$ for both $p=1$ and $p=\infty$ in proofs of Theorems~\ref{thm:general-2-3-bounds} and \ref{thm:4-flow-2d-bound}, in two-dimensional space the case $p=\infty$ can be deduced directly from $p=1$ using the equivalence below. The following observation is also independently observed in \cite{gaborik2025manhattan}.

\begin{observation}
\label{ob:2d-1-inf-equiv}
Let $G$ be a bridgeless graph. Then $G$ admits an $(r,2,1)$-NZF if and only if it admits an $(r,2,\infty)$-NZF. Consequently,
\[
\phi_{2,1}(G) = \phi_{2,\infty}(G).
\]
\end{observation}

\begin{proof}
In $\mathbb{R}^2$, the $1$-norm annulus
\[
\{\bs x \in \mathbb{R}^2 : 1 \le \|\bs x\|_1 \le r-1\}
\]
and the $\infty$-norm annulus
\[
\{\bs x \in \mathbb{R}^2 : 1 \le \|\bs x\|_\infty \le r-1\}
\]
are both square-shaped regions. More precisely, the $1$-norm annulus can be mapped onto the $\infty$-norm annulus by a $45^\circ$ rotation about the origin followed by a scaling by $\sqrt{2}$.
\end{proof}

In the remainder of this section, we prove Theorem~\ref{thm:4-flow-2d-bound}.
To establish the sharpness of the bound in Theorem~\ref{thm:4-flow-2d-bound},
we will use the following lemma comparing flow indices under different norms.
This lemma is also proved independently in \cite{Rieg-Master}. Note that for an $(\R^d\setminus\{0\})$-flow, the constraint $1 \le \lVert f(e) \rVert_p \le r-1$ in Definition~\ref{def:annular-flow} is equivalent to
\[
1 \le \frac{\max \lVert f(e) \rVert_p}{\min \lVert f(e) \rVert_p} \le r-1.
\]

\begin{lemma}\label{lem:norm-comparison}
Let $G$ be a bridgeless graph and $d \ge 1$ an integer. 
For any $p_1 < p_2 \in [1,\infty]$, we have
\[
d^{-(\frac{1}{p_1}-\frac{1}{p_2})} (\phi_{d,p_2}(G)-1) \leq \phi_{d,p_1}(G)-1
   \;\le\;
d^{\frac{1}{p_1}-\frac{1}{p_2}}\,(\phi_{d,p_2}(G)-1).
\]
\end{lemma} 
\begin{proof}
Note that  for every $\bs x \in \mathbb{R}^d$,
\[
\|\bs x\|_{p_2} \le \|\bs x\|_{p_1} \le d^{\frac{1}{p_1}-\frac{1}{p_2}}\|\bs x\|_{p_2}.
\]
Hence, for any $(\mathbb{R}^d \setminus \{0\})$-flow $(D,f)$ of $G$, we have
\begin{equation}\label{eq:two-sided}
\frac{\max_{e} \|f(e)\|_{p_2}}{d^{\frac{1}{p_1}-\frac{1}{p_2}} \min_{e} \|f(e)\|_{p_2}}
   \;\le\;
   \frac{\max_{e} \|f(e)\|_{p_1}}{\min_{e} \|f(e)\|_{p_1}}
   \;\le\;
   \frac{d^{\frac{1}{p_1}-\frac{1}{p_2}} \max_{e} \|f(e)\|_{p_2}}{\min_{e} \|f(e)\|_{p_2}}.
\end{equation}

For each $p \in [1,\infty]$, define
\[
\phi_p(f)=\frac{\max_{e} \|f(e)\|_{p}}{\min_{e} \|f(e)\|_{p}}.
\]
Then $\phi_{d,p}(G) \leq \phi_p(f)$. 
Since $\phi_{d,p}(G)$ is a minimum,  let $f_i$ satisfy $\phi_{p_i}(f_i)=\phi_{d,p_i}(G)-1$ for  each $i=1,2$.   It then follows that $\phi_{p_i}(f_i) \leq \phi_{p_i}(f_j)$ where $\{i,j\} = \{1,2\}$.

Applying the right-hand inequality in \eqref{eq:two-sided} to $f_2$ gives
\[
\phi_{d,p_1}(G)-1 = \phi_{p_1}(f_1)
   \le \phi_{p_1}(f_2)
   \le d^{\frac{1}{p_1}-\frac{1}{p_2}} \phi_{p_2}(f_2)
   = d^{\frac{1}{p_1}-\frac{1}{p_2}}\,(\phi_{d,p_2}(G)-1).
\]

Similarly, applying the left-hand inequality in \eqref{eq:two-sided} to $f_1$ yields
\[
d^{-(\frac{1}{p_1}-\frac{1}{p_2})} (\phi_{d,p_2}(G)-1)
   = d^{-(\frac{1}{p_1}-\frac{1}{p_2})} \phi_{p_2}(f_2)
   \le d^{-(\frac{1}{p_1}-\frac{1}{p_2})} \phi_{p_2}(f_1)
   \le \phi_{p_1}(f_1)
   = \phi_{d,p_1}(G)-1,
\]
which implies $\phi_{d,p_2}(G)-1 \le d^{\frac{1}{p_1}-\frac{1}{p_2}} (\phi_{d,p_1}(G)-1)$.  
This completes the proof.
\end{proof}

We are now ready to prove Theorem~\ref{thm:4-flow-2d-bound}.

\medskip\noindent
\textbf{Proof of Theorem~\ref{thm:4-flow-2d-bound}.}
Since $G$ admits a $4$-NZF, Lemma~\ref{thm:234-flow}(3) implies that $G$ admits a
$\{P_1,P_2,P_1+P_2,P_1-P_2\}$-flow for any choice of vectors
$P_1,P_2\in\mathbb{R}^2$.
It therefore suffices to choose $P_1$ and $P_2$ such that all vectors in
\[
\Omega=\{P_1,P_2,P_1+P_2,P_1-P_2\}
\]
have $p$-norms contained in $[1,r-1]$.
Lemma~\ref{lem:Omega-flow-phi} then yields $\phi_{2,p}(G)\le r$.

\medskip
\noindent\textbf{Case 1: $1\le p\le 2$.}

Let $P_1=(a,a)$ and $P_2=(a,-a)$, where $a=2^{-1/p}$.
Then $\|P_1\|_p=\|P_2\|_p=1$, and
\[
\|P_1\pm P_2\|_p=\|(2a,0)\|_p=\|(0,2a)\|_p=2^{\,1-1/p}.
\]
Thus every vector in $\Omega$ has $p$-norm in $[1,2^{\,1-1/p}]$, and hence
\[
\phi_{2,p}(G)\le 1+2^{\,1-1/p}.
\]

\medskip
\noindent\textbf{Case 2: $2\le p\le\infty$.}

Let $P_1=(1,0)$ and $P_2=(0,1)$.
Then $\|P_1\|_p=\|P_2\|_p=1$ and $\|P_1\pm P_2\|_p=2^{1/p}$, which implies
\[
\phi_{2,p}(G)\le 1+2^{1/p}.
\]

Together, Cases~1 and~2 establish the stated upper bounds.

\medskip
We now prove that these bounds are sharp.
By \cite[Theorem~1]{mattiolo2024lower}, $\phi_{2,2}(K_4)=1+\sqrt{2}$.
If a graph $G'$ admitting a $4$-NZF can be reduced to $K_4$ by identifying vertices
and deleting the resulting loops, then
\[
\phi_{2,2}(G')\ge \phi_{2,2}(K_4)=1+\sqrt{2}.
\]
Applying Lemma~\ref{lem:norm-comparison}, we obtain
\[
\phi_{2,p}(G') \ge
1 + 2^{-(\frac{1}{\min\{2,p\}}-\frac{1}{\max\{2,p\}})}\bigl(\phi_{2,2}(G')-1\bigr)
\ge
\begin{cases}
1 + 2^{\,1 - 1/p}, & 1 \le p \le 2,\\[1mm]
1 + 2^{\,1/p}, & 2 \le p \le \infty.
\end{cases}
\]

Combining this with the previously established upper bounds, we conclude that
$\phi_{2,p}(G')$ attains the values stated in the theorem.
Therefore, the bound in Theorem~\ref{thm:4-flow-2d-bound} is tight.
\hfill$\Box$

\section{Upper bounds for general $p$-norms}

In this section, we apply the approach introduced in Section~2 to prove Theorems~\ref{thm:general-2-3-bounds} and \ref{thm:3-dim-p-12-flow}.


\subsection{Proof of Theorem~\ref{thm:general-2-3-bounds}}
For the sake of convenience we recall Theorem~\ref{thm:general-2-3-bounds} first.

\begin{manualthm}[\ref{thm:general-2-3-bounds}]
Let $G$ be a bridgeless graph. Then we have
\begin{enumerate}[label=(\arabic*)]
    \item $\phi_{2,p}(G) \le 3$ for every $p\in[1,\infty]$;
    \item $\phi_{3,p}(G) \le 1+2^{1/p}$ for $p\in[1,\infty]$, and in particular, $\phi_{3,\infty}(G)=2$;
    \item $\phi_{3,1}(G) \le \dfrac{9}{4}$.
\end{enumerate}
\end{manualthm}


\begin{proof}

For $1\le p\le\infty$, write $a(p)=(2^p-1)^{1/p}$ and abbreviate it as $a$.
Then $a(\infty)=\lim_{p\to\infty}a(p)=2$, and $a^p=2^p-1$ for $p<\infty$.
Table~\ref{tab:three_schemes} lists feasible choices of vectors 
$P_1,P_2,P_3$ in the required dimensions and norms; these serve as our choices
of $\Omega$ in each case.

\begin{table}[htbp!]
\centering
\small
\renewcommand{\arraystretch}{1.3}
\begin{tabular}{M{2.5cm}  M{2.8cm}  M{2.8cm}  M{2.8cm} M{2.8cm}}
\toprule
\textbf{Vector} & $d=2,\,1\le p\le\infty$ & $d=3,\,1\le p\le\infty$ & $d=3,\,p=\infty$ & $d=3,\,p=1$ \\
\midrule
$P_1$ & $(0,-a)$ & $(0,0,1)$ & $(1,0,0)$ & $\left( -\frac{1}{2}, 0, \frac{1}{2} \right)$  \\
$P_2$ & $\Big(-\frac12,\frac a2\Big)$ & $(0,\frac a2,-\frac12)$ & $(0,1,0)$  & $\left( \frac{1}{8}, -\frac{1}{2}, -\frac{3}{8} \right)$\\
$P_3$ & $(1,0)$ & $(1,0,0)$ & $(0,0,1)$  & $\left( \frac{3}{8}, -\frac{1}{4}, \frac{3}{8} \right)$\\
$P_1+P_2$ & $\Big(-\frac12,-\frac a2\Big)$ & $(0,\frac a2,\frac12)$ & $(1,1,0)$  & $\left( -\frac{3}{8}, -\frac{1}{2}, \frac{1}{8} \right)$\\
$P_1+P_3$ & $(1,-a)$ & $(1,0,1)$ & $(1,0,1)$  & $\left( -\frac{1}{8}, -\frac{1}{4}, \frac{7}{8} \right)$ \\
$P_1-P_3$ & $(-1,-a)$ & $(-1,0,1)$ & $(1,0,-1)$ & $\left( -\frac{7}{8}, \frac{1}{4}, \frac{1}{8} \right)$ \\
$P_2+P_3$ & $\Big(\frac12,\frac a2\Big)$ & $(1,\frac a2,-\frac12)$ & $(0,1,1)$  & $\left( \frac{1}{2}, -\frac{3}{4}, 0 \right)$\\
$P_2-P_3$ & $\Big(-\frac32,\frac a2\Big)$ & $(-1,\frac a2,-\frac12)$ & $(0,1,-1)$ & $\left( -\frac{1}{4}, -\frac{1}{4}, -\frac{3}{4} \right)$  \\
$P_1+P_2+P_3$ & $\Big(\frac12,-\frac a2\Big)$ & $(1,\frac a2,\frac12)$ & $(1,1,1)$  & $\left( 0, -\frac{3}{4}, \frac{1}{2} \right)$\\
$P_1+P_2-P_3$ & $\Big(-\frac32,-\frac a2\Big)$ & $(-1,\frac a2,\frac12)$ & $(1,1,-1)$ & $\left( -\frac{3}{4}, -\frac{1}{4}, -\frac{1}{4} \right)$ \\
\bottomrule
\end{tabular}
\caption{Feasible vector assignments $P_1,P_2,P_3$ for different dimensions and norms.}
\label{tab:three_schemes}
\end{table}

\medskip
\noindent
\textbf{(1) $\boldsymbol{\phi_{2,p}(G)\le 3}$ for all $p\in[1,\infty]$.}

\medskip
Since $\phi_{2,1}(G)=\phi_{2,\infty}(G)$, it suffices to treat $p<\infty$.
Using the vectors in Column~2 of Table~\ref{tab:three_schemes}, it remains to
verify the inequalities
\begin{align}
& 1 \le |a|^p \le 2^p, \label{ineq:1}\\
& 1 \le \frac{1 + |a|^p}{2^p} \le 2^p, \label{ineq:2}\\
& 1 \le 1 + |a|^p \le 2^p, \label{ineq:3}\\
& 1 \le \frac{3^p + |a|^p}{2^p} \le 2^p. \label{ineq:4}
\end{align}
Substituting $a^p = 2^p - 1$, Inequalities~\eqref{ineq:1}--\eqref{ineq:3} and the left-hand side of~\eqref{ineq:4} follow immediately.  
For the right-hand side of~\eqref{ineq:4}, we need to show
\[
\frac{3^p + a^p}{2^p} \le 2^p
\quad \Longleftrightarrow \quad
3^p + 2^p - 1 \le 4^p.
\]
Equivalently, it suffices to prove
\[
4^p - 3^p \;\ge\; 2^p - 1.
\]
Since $f(x)=x^p$ is convex for every $p \ge 1$, we obtain
\[
\frac{4^p - 3^p}{4 - 3}
\;\ge\;
\frac{2^p - 1}{2 - 1}
\quad \Longleftrightarrow \quad
4^p - 3^p \ge 2^p - 1,
\]
which establishes the desired inequality.  
Therefore all constraints are satisfied, and we conclude that $\phi_{2,p}(G) \le 3$.

\medskip
\noindent
\textbf{(2) $\boldsymbol{\phi_{3,p}(G)\le 1+2^{1/p}}$ for $p\in[1,\infty]$.  
In particular, $\phi_{3,\infty}(G)=2$.}

\medskip
For $p=\infty$, besides the general construction in Column~3 of Table~\ref{tab:three_schemes}, we also present the simpler construction in Column~4.
Both constructions produce vectors of $\infty$-norm~$1$, and therefore $\phi_{3,\infty}(G)=2$.

For $p<\infty$, using Column~3, it suffices to check
\begin{align*}
1 &\le \frac{1+|a|^p}{2^p} \le 2, \\
1 &\le 1+\frac{1+|a|^p}{2^p} \le 2,
\end{align*}
which follow immediately from $a^p=2^p-1$.  
Hence $\phi_{3,p}(G)\le 1+2^{1/p}$.

\medskip
\noindent
\textbf{(3) $\boldsymbol{\phi_{3,1}(G)\le \frac94}$.}

\medskip
From Column~5 of Table~\ref{tab:three_schemes}, the maximum $1$-norm among the
listed vectors is $5/4$.  Therefore the corresponding $\Omega$ yields
$\phi_{3,1}(G)\le 9/4$. This completes the proof of Theorem~\ref{thm:general-2-3-bounds}.
\end{proof}

\subsection{Proof of Theorem~\ref{thm:3-dim-p-12-flow}}

We begin by recalling the functions introduced in Section~1:
\[
g_1(p) = \frac{5}{4} \cdot 2^{\,1 - 1/p}, \qquad
g_2(p) = \Bigl[(2^{-1/p} + 3^{-1/p})^p + (2^{-1/p} - 3^{-1/p})^p + \tfrac{1}{3}\Bigr]^{1/p}, \qquad
g_3(p) = 2^{1/p}.
\]
To establish Theorem~\ref{thm:3-dim-p-12-flow}, we proceed by proving the following two intermediate results.

\begin{theorem}\label{thm:phi-3-12-g1g2}
    Let $G$ be a bridgeless graph. Then for every $p\in[1,2]$,
\[
    \phi_{3,p}(G)\le 1+\min\{g_1(p),\, g_2(p)\}.
\]
\end{theorem}

\begin{proof}
By Theorem~\ref{thm:general-2-3-bounds}(3), we have $\phi_{3,1}(G)\le1+\frac54$. 
Applying Lemma~\ref{lem:norm-comparison} then yields
\[
    \phi_{3,p}(G)\le 1+2^{\,1-1/p}\bigl(\phi_{3,1}(G)-1\bigr)
    =1+\frac54\cdot 2^{\,1-1/p}
    =1+g_1(p).
\]
It therefore remains to prove that $\phi_{3,p}(G)\le 1+g_2(p)$.

Define
\[
    a(p)=2^{-1/p},
    \qquad
    b(p)=3^{-1/p},
\]
and write $a$ and $b$ for brevity.  
Consider the vectors
\[
    P_1=(a,a,0),\qquad
    P_2=(-a,0,-a),\qquad
    P_3=(b,-b,-b).
\]
The remaining vectors in the set $\Omega$ (defined in Theorem~\ref{thm:6-flow-points}) are
\[
\begin{aligned}
    P_1+P_2 &= (0,a,-a),\\
    P_1+P_3 &= (a+b,\,a-b,\,-b),\\
    P_1-P_3 &= (a-b,\,a+b,\,b),\\
    P_2+P_3 &= (-a+b,\,-b,\,-a-b),\\
    P_2-P_3 &= (-a-b,\,b,\,-a+b),\\
    P_1+P_2+P_3 &= (\,b,\,a-b,\,-a-b),\\
    P_1+P_2-P_3 &= (-b,\,a+b,\,-a+b).
\end{aligned}
\]

We now verify that each vector has $p$-norm between $1$ and $g_2(p)$.  
Since $2a^p=1$ and $3b^p=1$, we obtain
\[
    \|P_1\|_p^p=\|P_2\|_p^p=\|P_1+P_2\|_p^p=2a^p=1,
    \qquad
    \|P_3\|_p^p=3b^p=1.
\]
All remaining vectors have the same $p$-norm as $P_1+P_3$, namely
\[
    \|P_1+P_3\|_p
    =\bigl[(a+b)^p+(a-b)^p+b^p\bigr]^{1/p}
    =g_2(p).
\]

Finally, since
\[
    g_2(p)
    \ge \bigl[(a+b)^p + b^p\bigr]^{1/p}
    \ge \bigl[(2b)^p + b^p\bigr]^{1/p}
    =\Bigl(\frac{2^p+1}{3}\Bigr)^{1/p}
    \ge 1,
\]
all vectors in $\Omega$ indeed have $p$-norm at least $1$ and at most $g_2(p)$.
This completes the proof.
\end{proof}

\begin{theorem}\label{thm:phi3sqrt2}
    For every bridgeless graph $G$ and every $p\in[1,\infty]$,
\[
\phi_{3,p}(G)\le1+\sqrt2.
\]
\end{theorem}

\begin{proof}
By Theorem~\ref{thm:general-2-3-bounds}(2), we have $\phi_{3,p}(G)\le 1+2^{1/p}\le 1+\sqrt2$ for all $p\in[2,\infty]$. Thus it remains to consider $p\in[1,2]$. By Theorem~\ref{thm:phi-3-12-g1g2}, for every $p\in[1,2]$,
\[
\phi_{3,p}(G)\le 1+\min\{g_1(p),g_2(p)\},
\]
so it suffices to show that $g_2(p)\le\sqrt2$ for all $p\in[1,2]$. Recall that
\[
g_2(p)
=\Big[(2^{-1/p}+3^{-1/p})^p + (2^{-1/p}-3^{-1/p})^p + \tfrac13\Big]^{1/p}.
\]

\medskip
We first handle the range $p\in[1.6,2]$. Set
\[
I(p)
=\Big[(2^{-1/p}+3^{-1/p})^2+(2^{-1/p}-3^{-1/p})^2\Big]
-\Big[(2^{-1/p}+3^{-1/p})^p+(2^{-1/p}-3^{-1/p})^p\Big].
\]
Let
\[
x(p)=2^{-1/p}+3^{-1/p},\qquad y(p)=2^{-1/p}-3^{-1/p}.
\]
Then $x(p)$ is increasing on $[1,2]$ and, for all $p\in[1.6,2]$, $x(p)\ge x(1.6)>1$. Moreover $y(p)$ is decreasing on $[1,2]$ and satisfies $0<y(p)\le y(1)=1/6<1$. A direct numerical evaluation shows that
\[
x(1.6)^{1.6}\ln x(1.6)+y(1.6)^{1.6}\ln y(2)>0.08>0.
\]
Using the identities
\[
z^2-z^p=\int_p^2 z^t\ln z\,dt,
\]
we obtain
\[
\begin{aligned}
I(p)
&=(x(p)^2-x(p)^p)+(y(p)^2-y(p)^p) \\
&=\int_p^2 x(p)^t\ln x(p)\,\mathrm{d}t+\int_p^2 y(p)^t\ln y(p)\,\mathrm{d}t\\
&=\int_p^2\!\Big[x(p)^t\ln x(p)-y(p)^t\ln\tfrac1{y(p)}\Big]\,\mathrm{d}t.
\end{aligned}
\]
Since $x(p)$ is increasing and $y(p)$ is decreasing on $[1.6,2]$, we have
\[
\begin{aligned}
I(p)
&\ge \int_p^2\!\Big[x(1.6)^{1.6}\ln x(1.6)-y(1.6)^{1.6}\ln\tfrac1{y(2)}\Big]\,\mathrm{d}t\\
&=(2-p)\Big[x(1.6)^{1.6}\ln x(1.6)+y(1.6)^{1.6}\ln y(2)\Big]\ge0.
\end{aligned}
\]
Hence $I(p)\ge0$ for $p\in[1.6,2]$, and consequently
\[
\begin{aligned}
g_2(p)
&=\big(x(p)^p+y(p)^p+\tfrac13\big)^{1/p}
\le\big(x(p)^2+y(p)^2+\tfrac13\big)^{1/p}\\
&=\big(2(2^{-2/p}+3^{-2/p})+\tfrac13\big)^{1/p}.
\end{aligned}
\]
Thus it remains to verify the inequality
\[
H_1(p):=2\big(2^{-2/p}+3^{-2/p}\big)+\tfrac13-2^{p/2}\le0
\qquad (p\in[1.6,2]).
\]
Let $m=-2/p$. Then $p\in[1.6,2]$ corresponds to $m\in[-5/4,-1]$, and
\[
H_1(p)=H_2(m):=2(2^m+3^m)+\tfrac13-2^{-1/m}.
\]
Since $-1/m\ge m+2$ on $[-5/4,-1]$, we obtain
\[
H_2(m)\le 2(2^m+3^m)+\tfrac13-2^{m+2}=2(3^m-2^m)+\tfrac13.
\]
Define
\[
h(m)=2(3^m-2^m)+\tfrac13.
\]
The function $h$ is strictly decreasing on $(-\infty,m_0]$ and strictly increasing on $[m_0,\infty)$, where
\[
m_0=\log_{3/2}\log_2 3
\]
is its unique minimizer. Since
\[
h(-5/4)\le -0.001<0,\qquad h(-1)=0,
\]
we have $h(m)\le0$ for all $m\in[-5/4,-1]$, which implies $H_1(p)\le0$, and therefore $g_2(p)\le\sqrt2$ on $[1.6,2]$.

\medskip
We now consider $p \in [1,1.6]$.  
For any $1 \le p_1 \le p \le p_2 \le 1.6$, we can rewrite $g_2(p)$ as
\[
\begin{aligned}
g_2(p)
&= \Bigl[(2^{-1/p}+3^{-1/p})^p + (2^{-1/p}-3^{-1/p})^p + \tfrac13 \Bigr]^{1/p} \\
&= \left[ \frac{(r^{1/p}+1)^p + (r^{1/p}-1)^p + 1}{3} \right]^{1/p},
\end{aligned}
\]
where $r = 2/3$. Define
\[
F_1(p) = (r^{1/p}+1)^p, \qquad F_2(p) = (r^{1/p}-1)^p.
\]
Differentiation shows that
\[
F_1'(p) = \frac{F_1(p)\,\ln r}{p} \Biggl(\frac{\ln(r^{1/p}+1)^p}{\ln r} - \frac{r^{1/p}}{r^{1/p}+1}\Biggr)
\ge \frac{F_1(p)\,\ln r}{p} \Biggl(\frac{\ln(1+r)}{\ln r} - \frac{r^{1/p}}{r^{1/p}+1}\Biggr) \ge 0,
\]
and
\[
F_2'(p) = \frac{F_2(p)\,\ln r}{p} \Biggl(\frac{\ln(r^{1/p}-1)^p}{\ln r} - \frac{r^{1/p}}{r^{1/p}-1}\Biggr)
< \frac{F_2(p)\,\ln r}{p} \Biggl(1 - \frac{r^{1/p}}{r^{1/p}-1}\Biggr) < 0.
\]
Hence, $F_1$ is increasing and $F_2$ is decreasing on $[1,1.6]$.  
It follows that, on each subinterval $[p_1,p_2]$,
\[
g_2(p) = \left[ \frac{F_1(p) + F_2(p) + 1}{3} \right]^{1/p} \le \left[ \frac{F_1(p_2) + F_2(p_1) + 1}{3} \right]^{1/p_1}.
\]

Define
\[
\Phi(p_1,p_2) := \left[ \frac{F_1(p_2) + F_2(p_1) + 1}{3} \right]^{1/p_1}.
\]
By direct computation over consecutive subintervals of $[1,1.6]$, we obtain
\[
\begin{aligned}
\Phi(1,1.1) &= 1.391508229\ldots,\\
\Phi(1.1,1.25) &= 1.408323784\ldots,\\
\Phi(1.25,1.4) &= 1.403094249\ldots,\\
\Phi(1.4,1.5) &= 1.384104551\ldots,\\
\Phi(1.5,1.6) &= 1.393230647\ldots.
\end{aligned}
\]

On each subinterval, the monotonicity of the terms ensures that $g_2(p) \le \Phi(p_1,p_2)$ for all $p$ in that subinterval.  
Since all computed values of $\Phi(p_1,p_2)$ are less than $\sqrt{2} \approx 1.4142$, we conclude that
\[
g_2(p) \le \sqrt{2} \quad \text{for every } p \in [1,1.6].
\]

\medskip
Combining the two ranges  $[1,1.6]$ and $[1.6,2]$ completes the proof that $g_2(p)\le\sqrt2$ for all $p\in[1,2]$. This completes the proof.
\end{proof}

Combining Theorems~\ref{thm:phi-3-12-g1g2}, \ref{thm:phi3sqrt2}, and \ref{thm:general-2-3-bounds}(2), Theorem~\ref{thm:3-dim-p-12-flow} follows directly as desired.

\section{Upper bounds under oriented $k$-cycle $2l$-covers}

In this section, we apply the approach introduced in Section~2 to prove
Theorem~\ref{thm:intro-k-OCDC-phi}. We begin by establishing an auxiliary result
that provides the foundation for deriving general upper bounds.

\begin{theorem}\label{thm:k-OCDC-points}
Let $G$ be a bridgeless graph, and let $d \ge 2$, $k \ge 2$, and $l \ge 1$ be integers.  
Suppose that $G$ admits an oriented $k$-cycle $2l$-cover
\[
\mathcal{C} = \{C_1, C_2, \dots, C_k\}.
\]
Then, for any vectors $P_1, P_2, \dots, P_k \in \mathbb{R}^d$, the graph $G$ admits
an $\Omega$-flow, where
\[
\Omega =
\Bigl\{
\sum_{i \in I} P_i - \sum_{j \in J} P_j :
I, J \subseteq \{1,2,\dots,k\},\ |I| = |J| = l,\ I \cap J = \varnothing
\Bigr\}.
\]
\end{theorem}

\begin{proof}
For each directed cycle $C_i \in \mathcal{C}$, assign the vector $P_i \in \mathbb{R}^d$ to every edge of $C_i$.  
Since $C_i$ is a directed cycle, this assignment defines a $\{P_i\}$-flow supported on $C_i$.

Let $C_i(e)$ denote the orientation of $e \in E(G)$ in the directed cycle $C_i$.  
Fix an arbitrary orientation $D$ of $G$, and for each edge $e \in E(G)$, define
\[
I(e) = \{\, i : C_i(e) = D(e) \,\}, 
\qquad 
J(e) = \{\, i : C_i(e) \neq D(e) \,\}.
\]
Because $\mathcal{C}$ is an oriented $k$-cycle $2l$-cover, each edge $e$ appears in exactly $l$ cycles in each direction.  
Hence, $|I(e)| = |J(e)| = l$ and $I(e) \cap J(e) = \varnothing$.

Define a function $f : E(G) \to \mathbb{R}^d$ by
\[
f(e) = \sum_{i \in I(e)} P_i - \sum_{j \in J(e)} P_j.
\]
By construction, $f(e) \in \Omega$ for every $e \in E(G)$.  
Moreover, since $f$ is obtained as a linear combination of the flows on cycles with coefficients $\pm 1$, it satisfies flow conservation at every vertex.  
Therefore, $(D,f)$ is an $\Omega$-flow of $G$.
\end{proof}

We are now ready to prove Theorem~\ref{thm:intro-k-OCDC-phi}. We first recall its statement.

\begin{manualthm}[\ref{thm:intro-k-OCDC-phi}]
Let $G$ be a graph that admits a $(k,2l)$-OCC. Then we have
\begin{enumerate}[label=(\arabic*)]
    \item for every $p \in [1,\infty]$, $\phi_{d,p}(G) = 2$, where $d = k-1$;
    \item when $l = 1$, $\phi_{d,1}(G) = 2$, where $d = \lceil k/2 \rceil$.
\end{enumerate}
\end{manualthm}

\begin{proof}
(1) Let  $\{e_1, \dots, e_{d+1}\}$  be the standard basis of $\mathbb{R}^{d+1}$.  
Define
\[
v_i := e_i - \frac{1}{d+1} \mathbf{1}, \qquad  \mbox{for each}~~i=1,\dots,d+1,
\]
where $\mathbf{1} = (1,1,\dots,1) \in \mathbb{R}^{d+1}$.  

For any disjoint subsets $I, J \subseteq\{1,\dots,d+1\}$ with $|I| = |J| = l$, define
\[
v_{I,J} := \sum_{i \in I} v_i - \sum_{j \in J} v_j.
\]
Then $v_{I,J}$ has exactly $l$ coordinates equal to $1$, $l$ coordinates equal to $-1$, and the remaining coordinates equal to $0$.  
Hence, 
\[
\|v_{I,J}\|_p = (2l\cdot1^p)^{1/p} = (2l)^{1/p} \quad \text{for } p \ge 1, 
\qquad \mbox{and}~~
\|v_{I,J}\|_\infty = 1.
\]

Set the scaling factor
\[
\alpha_p := 
\begin{cases}
(2l)^{1/p}, & 1 \le p < \infty,\\[1mm]
1, & p = \infty,
\end{cases}
\]
and set 
\[
s_i := \frac{v_i}{\alpha_p}, \quad i=1,\dots,d+1.
\]
Then for any $I,J$ as above,
\[
\Bigl\| \sum_{i \in I} s_i - \sum_{j \in J} s_j \Bigr\|_p = 1 \quad \text{for all } p \in [1,\infty].
\]
All vectors $s_i$ lie in the hyperplane
\[
H := \left\{ x=(x_1,\dots,x_{d+1}) \in \mathbb{R}^{d+1} : \sum_{l=1}^{d+1} x_l = 0 \right\} \cong \mathbb{R}^d.
\]

By Theorem~\ref{thm:k-OCDC-points}, assigning $P_i := s_i$ to the $i$th cycle of an oriented $k$-cycle $2l$-cover yields an $\Omega$-flow of $G$. Although the vectors $s_i$ are defined in $\mathbb{R}^{d+1}$, they all lie in the $d$-dimensional hyperplane $H$, and hence the resulting $\Omega$-flow is realized in $\mathbb{R}^d$. Moreover, every vector in $\Omega$ has $p$-norm equal to $1$. Therefore, this construction produces a feasible $\Omega$-flow for $G$, and we conclude that
\[
\phi_{d,p}(G) = 2 \quad \text{for all } p \ge 1 \text{ or } p = \infty.
\]

(2)   
Define $k$ vectors $P_1, \dots, P_k \in \mathbb{R}^d$ as
\[
P_i = \frac{1}{2} (p_{i,1}, \dots, p_{i,d}), \quad i=1,\dots,k,
\]
where
\[
p_{i,j} =
\begin{cases}
(-1)^i, & j = \lceil i/2 \rceil,\\
0, & \text{otherwise.}
\end{cases}
\]

In other words, the nonzero entry of $P_i$ is in coordinate $\lceil i/2 \rceil$, with value $+1/2$ if $i$ is even and $-1/2$ if $i$ is odd; all other coordinates are zero.  

Then for any distinct $i,j \in [k]$, $P_i - P_j$ has either one nonzero coordinate equal to $\pm 1$, or two nonzero coordinates each equal to $\pm 1/2$, so that
\[
\|P_i - P_j\|_1 = 1.
\]

Assign these vectors to the cycles of the $(k,2)$-OCC according to Theorem~\ref{thm:k-OCDC-points}.  
Each edge then receives a flow
\[
f(e) = P_i - P_j \in \Omega,
\]
with $\Omega$ as in Theorem~\ref{thm:k-OCDC-points}, and the $p$-norm of each vector in $\Omega$ is exactly $1$.  

Hence,
\[
\phi_{d,1}(G) = 2.
\]

This completes the proof.
\end{proof}

\section{Discussions and Open Problems}
\subsection{Remarks on Theorem~\ref{thm:3-dim-p-12-flow}}

For convenience, we recall the auxiliary functions:
\[
g_1(p) = \frac{5}{4} \cdot 2^{\,1 - 1/p}, \qquad
g_2(p) = \Bigl[(2^{-1/p} + 3^{-1/p})^p + (2^{-1/p} - 3^{-1/p})^p + \tfrac{1}{3}\Bigr]^{1/p}, \qquad
g_3(p) = 2^{1/p}.
\]
Theorem~\ref{thm:3-dim-p-12-flow} establishes that
\(
\phi_{3,p}(G) \le 1 + \min\{g_1(p), g_2(p),g_3(p)\}\le 1+\sqrt2
\)
for all $p \in [1,\infty]$.

Figure~\ref{fig:three curves} illustrates the behavior of the functions
$g_1$, $g_2$, and $g_3$.  
In particular, note that $g_1(1) = \tfrac{5}{4}$ and $g_2(2) = g_3(2) = \sqrt{2}$.  
Based on the plot, together with numerical verification,  
the curves of $g_1(p)$ and $g_2(p)$ appear to intersect at approximately
$p_0 \approx 1.0913$.  
This observation suggests the following piecewise upper bound:
\[
\phi_{3,p}(G) \le 
\begin{cases}
1 + g_1(p), & p \in [1, p_0],\\
1 + g_2(p), & p \in [p_0, 2],\\
1 + g_3(p), & p \in [2, \infty].
\end{cases}
\]

We emphasize that this piecewise description is heuristic;  
no rigorous proof of the exact crossover behavior is provided here.  
Nevertheless, Theorem~\ref{thm:3-dim-p-12-flow} provides the uniform bound
$1 + \sqrt{2}$, for which a complete proof is given in this paper.

\begin{figure}[htbp]
    \centering
    \includegraphics[width=0.5\linewidth]{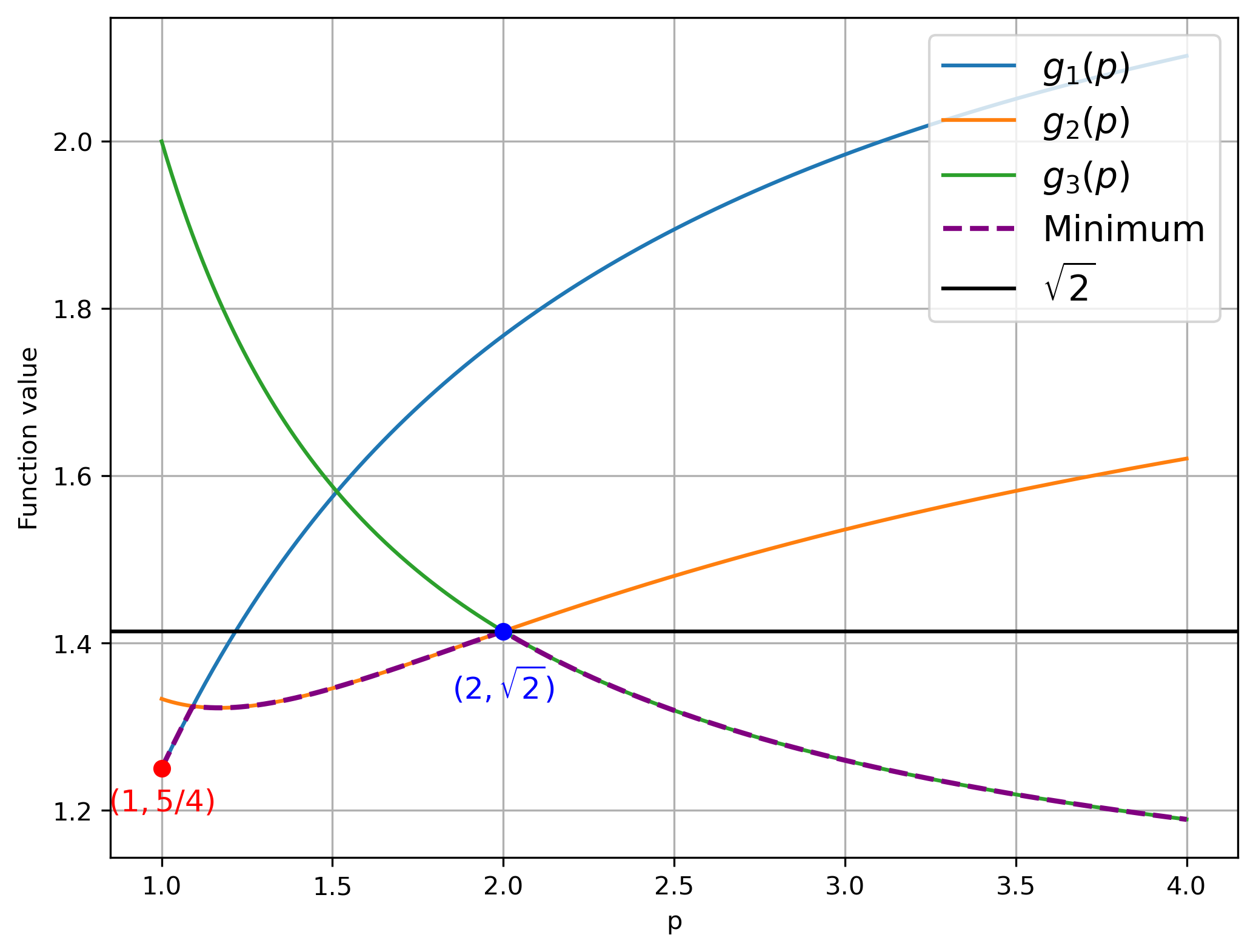}
    \caption{Curves of the functions $g_1(p)$, $g_2(p)$, and $g_3(p)$.}
    \label{fig:three curves}
\end{figure}

\subsection{Open problems and conjectures}

The results of this paper suggest several directions for further research at the intersection of combinatorial flow theory, geometry, and topology. 
We summarize below a few conjectures and problems that we consider particularly interesting. 

\medskip \noindent
{ \bf (a) Monotonicity and Duality of Flow Indices}
\medskip

Our computational experiments indicate that the $p$-normed flow index $\phi_{d,p}(G)$ exhibits a structured behavior and certain symmetries as a function of the norm parameter $p$. 

\begin{conjecture}[Monotonicity of $\phi_{d,p}(G)$]
\label{conj:monotonicity}
Let $G$ be a bridgeless graph and $d \ge 2$. 
As a function of $p$, the quantity $\phi_{d,p}(G)$ is continuous, increasing on $[1,2]$, and decreasing on $[2,\infty]$.
\end{conjecture}

In communication with Mazzuoccolo, we learned that Rieg's Master's Thesis \cite{Rieg-Master} proves the continuity of $\phi_{d,p}(G)$ as a function of $p$. They also conjectured a similar monotonicity property, stating that $\phi_{d,p}(G)$ is increasing on $[1,p_0]$ and decreasing on $[p_0,\infty]$ for some $p_0$, and we conjecture that $p_0 = 2$.

\begin{conjecture}[Duality between norms]
\label{conj:duality}
Let $G$ be a  bridgeless graph  and $d \ge 2$ be an integer. Then  for any $p \in [1,2]$ there exists $q \in [2,\infty]$ such that
\[
\phi_{d,p}(G) = \phi_{d,q}(G),
\]
and conversely, for each $q \in [2,\infty]$ there exists $p \in [1,2]$ satisfying the same equality.
\end{conjecture}

In communication with Mazzuoccolo, we also learned that Rieg \cite{Rieg-Master} formulated a similar duality conjecture, where $q$ is taken as the usual $L^p$–$L^q$ dual, i.e., $1/p + 1/q = 1$. However, we believe that this exact duality may not hold in general; there could be slight deviations from the classical $L^p$–$L^q$ correspondence.

\medskip \noindent
{ \bf (b) Connections with Circular Flows}
 \medskip 
 
By Corollary~\ref{Cor:3-4-flow}, if a graph $G$ admits a $3$-NZF, then $\phi_{2,p}(G) = 2$ for all $p \in [1,\infty]$. 
In contrast, if $G$ admits a $4$-NZF, we only know that $\phi_{2,p}(G) = 2$ for $p = 1$ and $p = \infty$. 
For example, the complete graph $K_4$ admits a $4$-NZF but satisfies $\phi_{2,2}(K_4) = 1 + \sqrt{2} > 2$~\cite{mattiolo2023d}. 

These observations motivate the following conjecture in the context of classical circular flows.

\begin{conjecture}\label{conj:circular-NZF-AF}
For $\epsilon \in [0,1]$, there exist two functions $f$ and $g$, with $f(\epsilon)\in[2,+\infty]$ strictly increasing and $g(\epsilon)\in[1,2]$ strictly decreasing, such that if a graph $G$ admits a circular $(3+\epsilon)$-NZF, then
\[
\phi_{2,p}(G) = 2 \quad \text{for all } p \ge f(\epsilon) \text{ and for all } 1 \le p \le g(\epsilon).
\]
\end{conjecture}

\medskip \noindent
{ \bf (c) Connections with Oriented Cycle Double Covers}
\medskip 

The relationship between high-dimensional flows and oriented $k$-cycle double covers remains central to understanding the extremal behavior of $\phi_{d,p}(G)$. 
Our Theorem~\ref{thm:intro-k-OCDC-phi} shows that an oriented $k$-cycle double cover ensures $\phi_{d,p}(G)=2$ when $d=k-1$. 
This leads to the following broader question.

\begin{problem}
Characterize all graphs $G$ for which $\phi_{d,p}(G)=2$ holds for some pair $(d,p)$, even when $G$ does not admit an oriented $(d+1)$-cycle double cover.
\end{problem}

\medskip \noindent
{\bf (d) Toward the $S^2_p$-Flow Conjecture}

\medskip 

Denote $S_{p}^{d-1} := \{\bs{x} \in \R^d : \lVert \bs{x} \rVert_p = 1 \}$.  
Jain’s $S^2$-Flow Conjecture asserts that every bridgeless graph admits a unit-sphere flow in $\mathbb{R}^3$ (i.e., $\phi_{3,2}(G)=2$).  
Our results show that $\phi_{3,\infty}(G)=2$, and when combined with Conjectures~\ref{conj:monotonicity} and \ref{conj:duality}, they naturally suggest the following generalization.

\begin{conjecture}[$S^2_p$-Flow Conjecture]
For every bridgeless graph $G$ and every $p \in[1,\infty]$,
\[
\phi_{3,p}(G) = 2.
\]
\end{conjecture}

\section*{Acknowledgments}
The authors would like to thank Professor Giuseppe Mazzuoccolo for helpful suggestions and for drawing our attention to the related references \cite{gaborik2025manhattan, Rieg-Master}. Jiaao Li is partially supported by National Key Research and Development Program of China (No. 2022YFA1006400), National Natural Science Foundation of China (Nos. 12571371, 12222108), Natural Science Foundation of Tianjin (No. 24JCJQJC00130), and the Fundamental Research Funds for the Central Universities, Nankai University. Rong Luo is partially supported by a grant from  Simons Foundation (No. 839830).

\end{document}